\documentclass[graybox]{svmult}

\usepackage{mathptmx}       
\usepackage{helvet}         
\usepackage{courier}        
\usepackage{type1cm}        
\usepackage{makeidx}         
\usepackage{graphicx}        
\usepackage{multicol}        
\usepackage[bottom]{footmisc}

\usepackage{amsfonts}
\usepackage{psfrag}

\usepackage{amsmath}
\usepackage{amssymb}
\usepackage{graphicx}
\usepackage{bbm}
\usepackage{dsfont}

\newcommand{\norm}[1]{\left\Vert #1\right\Vert_2}
\newcommand{\einsnorm}[1]{\left\Vert #1\right\Vert_1}
\newcommand{\betrag}[1]{\left\vert #1\right\vert}
\newcommand{\klamm}[1]{\left( #1\right)}
\newcommand{\Klamm}[1]{\left[ #1\right]}
\newcommand{\KLAMM}[1]{\left\{ #1\right\}}
\newcommand{\Zett}{{\mathbb{Z}^d}}
\newcommand{\Zettp}{{\mathbb{Z}^{pd}}}
\newcommand{\Zetttwo}{{\mathbb{Z}^{2d}}}

\newcommand{\expo}{\mathrm{e}}
\newcommand{\solu}[2]{{u\left(#1,#2\right)}}
\newcommand{\rsol}[2]{{\tilde{u}\left(#1,#2\right)}}

\newcommand{\rmom}[2]{\tilde{m}_{#2}^p\klamm{#1}}  
\newcommand{\spitz}[1]{\langle #1 \rangle}

\newcommand{\Spectrum}[1]{\sigma\left({#1}\right)}
\newcommand{\PSpectrum}[1]{\sigma_p\left({#1}\right)}

\newcommand{\Hsym}[1]{\tilde{\mathcal{H}}^{#1}}
\newcommand{\Asym}[1]{\tilde{\mathcal{A}}^{#1}}
\newcommand{\Vsym}[1]{\tilde{V}^{#1}}
\newcommand{\Rsym}[1]{\tilde{R}_{#1}}

\newcommand{\Symm}[1]{\mathfrak{S}^{#1}}

\newcommand{\rk}[2]{r^{(#1)}_{#2}}

\newcommand{\TS}[1]{\tilde{T}_{#1}}
\newcommand{\TSi}[1]{\tilde{T}^{(#1)}}

\newcommand{\Z}{{\mathbb Z}}

\renewcommand{\P}{{\mathbb P}}
\renewcommand{\E}{{\mathbb E}}

\allowdisplaybreaks[1]

\makeindex             

\begin{document}

\title*{Precise asymptotics for the parabolic Anderson model with a moving catalyst or trap}

\author{Adrian Schnitzler and Tilman Wolff}

\institute{
Adrian Schnitzler \at Technische Universit\"at Berlin, Institut f\"ur Mathematik, Stra\ss e des 17 Juni 136, 10623 Berlin, Germany, \email{schnitzler@math.tu-berlin.de}
\and Tilman Wolff \at Weierstrass Institute for Applied Analysis and Stochastics, Mohrenstra\ss e 39, 10117 Berlin, Germany, \email{wolff@wias-berlin.de}}

\maketitle

\abstract{We consider the solution $u\colon [0,\infty) \times\mathbb{Z}^d\rightarrow [0,\infty) $ to the parabolic Anderson model, where the potential is given by $(t,x)\mapsto\gamma\delta_{Y_t}\left(x\right)$ with $Y$ a simple symmetric random walk on $\mathbb{Z}^d$. Depending on the parameter $\gamma\in[-\infty,\infty)$, the potential is interpreted as a randomly moving catalyst or trap.\\
In the trap case, i.e., $\gamma<0$, we look at the annealed time asymptotics in terms of the first moment of $u$. Given a localized initial condition, we derive the asymptotic rate of decay to zero in dimensions 1 and 2 up to equivalence and characterize the limit in dimensions 3 and higher in terms of the Green's function of a random walk. For a homogeneous initial condition we give a characterisation of the limit in dimension 1 and show that the moments remain constant for all time in dimensions 2 and higher.\\
In the case of a moving catalyst ($\gamma>0$), we consider the solution $u$ from the perspective of the catalyst, i.e., the expression $u(t,Y_t+x)$. Focusing on the cases where moments grow exponentially fast (that is, $\gamma$ sufficiently large), we describe the moment asymptotics of the expression above up to equivalence. Here, it is crucial to prove the existence of a principal eigenfunction of the corresponding Hamilton operator. While this is well-established for the first moment, we have found an extension to higher moments.
}

\section{Introduction}\label{Introduction}
The parabolic Anderson model (PAM) is the heat equation on the lattice with a random potential, given by
\begin{eqnarray}\label{AWPe}
\begin{cases}
\frac{\partial}{\partial t}u(t,x) = \kappa \Delta u(t,x)+ \xi(t,x) u(t,x),
 \qquad &(t,x)\in  (0,\infty) \times\mathbb{Z}^{d},\\ 
 u(0,x)=u_{0}(x), \qquad &x\in \mathbb{Z}^{d},
\end{cases}
\end{eqnarray}
where $\kappa>0$\index{$\kappa$} denotes a diffusion constant, $u_{0}$ a nonnegative function and $\Delta$ the discrete Laplacian, defined by	
	\[
\Delta f(x) := \sum\limits_{\substack{y\in\mathbb{Z}^{d}\colon\\|x-y|=1}}\left[f(y)-f(x)\right],\qquad x\in \mathbb{Z}^{d},\, f\colon\mathbb{Z}^{d}\rightarrow \mathbb{R}.
\]
Furthermore,
$\xi\colon [0,\infty) \times\mathbb{Z}^{d}\longrightarrow\mathbb{R}$
is a space and time dependent random potential.\\
We deal with the special case that the potential is given by
\[\xi(t,x)=\gamma\delta_{Y_{t}}(x),\qquad(t,x)\in\  [0,\infty) \times\mathbb{Z}^{d},
\]
with a simple symmetric random walk $Y$ with generator $\rho\Delta$ that starts in the origin and a parameter $\gamma\in[-\infty,\infty)$ called \emph{coupling constant}. In this paper we analyse the large time asymptotics after averaging over the potential which is usually referred to as \emph{annealed asymptotics}. We denote expectation with respect to the potential $\xi$ by $\left<\cdot\right>$.\\
One possible interpretation of this system arises from chemistry. Here, $u(t,x)$ describes the concentration of reactant particles in a point $x$ at time $t$ in presence of a randomly moving particle. In the case $\gamma<0$, the particle acts as a decatalyst (or trap) that kills reactant particles with rate $-\gamma$ at its position. In the case of positive $\gamma$, we consider a catalyst particle that causes reactants to multiply with rate $\gamma$. In both cases $\left<u(t,x)\right>$ is interpreted as the averaged concentration. For further interpretations and an overview over the PAM see for instance \cite{GM90}, \cite{CM94}, \cite{M94} and \cite{GK05}.\\
Annealed asymptotics in the case of a positive coupling constant $\gamma$ have already been investigated in \cite{GH06}. In the present work, we derive similar results with regard to the expression $\rsol{t}{x}:=\solu{t}{Y_t+x}$, which can be interpreted as the particle concentration in a neighbourhood of the catalyst. In addition to logarithmic asymptotics in terms of Lyapunov exponents, we derive asymptotics up to equivalence for most of the parameter choices where exponential growth is observed.\\
The case that $\gamma$ is negative has to the best of our knowledge not been investigated so far. Its analysis relies on techniques quite different from those in the catalyst case as a functional analytic approach proves unfeasible here. We calculate moment limits dependent on the model parameters and, in the case of moment convergence towards zero, specify the convergence speed up to equivalence.\\
Whereas the PAM with time independent potential or white-noise potential is well understood, some other time dependent potentials have just been examined recently.  In \cite{GdH06}, \cite{GdHM11} and \cite{KS03}, for instance, the authors investigate the case of infinitely many randomly moving catalysts. In \cite{CGM11} the authors deal with the case of finitely many catalysts, whereas the article \cite{DGRS11} is dedicated to a model similar to the case of infinitely many moving traps. Further examples of time dependent potentials can be found in \cite{GdHM07}, \cite{GdHM09}, \cite{GdHM10}, \cite{MMS11}, and the recent survey \cite{GdHM09b}. Within these proceedings, \cite{KS11}, \cite{LM11} and \cite{MZ11} deal with the parabolic Anderson model with time-independent potential.\\

\noindent
In Section~\ref{seclic} we analyze the PAM with localized initial condition $u_0=\delta_z$ and \mbox{$\gamma<0$}.
Let
\[
M_{z}(t):=\Big\langle \sum\limits_{x\in \mathbb{Z}^{d}}u(t,x)\Big\rangle,\qquad (t,z)\in  [0,\infty) \times\mathbb{Z}^{d},\]
denote the expected total mass of the system at time $t$ if the solution is initially localized in $z$ and the trap starts in the origin.
We find
 \begin{theorem}\label{locd1}
For $d=1,2$ and every $z\in\mathbb{Z}^{d}$,
\begin{eqnarray*}
	\left(i\right)\qquad M_{z}(t)&\sim \frac{2}{\sqrt{\pi}}\frac{\sqrt{\kappa+\rho}}{-\gamma}t^{-\frac{1}{2}},\qquad\qquad t\rightarrow\infty\qquad &\text{for}\qquad d=1;\\ 
	 \left(ii\right)\qquad M_{z}(t)&\sim 4\pi\frac{\kappa+\rho}{-\gamma}\left(\log t\right)^{-1},\qquad t\rightarrow\infty\qquad &\text{for}\qquad d=2,
\end{eqnarray*}
\end{theorem} 
and
\begin{theorem}\label{locd3}
For $d\geq3$ and every $z\in\mathbb{Z}^{d}$,
\[\lim_{t\rightarrow\infty}M_{z}(t)=1+\frac{\gamma}{\kappa+\rho-\gamma G_1(0)}G_1(z),
\]
\end{theorem}
where $G_\kappa$ denotes the Green's function of a random walk with generator $\kappa\Delta$.
\begin{remark}
Theorems~\ref{locd1} and \ref{locd3} can be generalized to all initial conditions with compact support without much effort.
\end{remark}

\noindent
In Section~\ref{sechom} we analyze the case of a homogeneous initial condition $u_0\equiv1$. We find that in dimensions 2 and higher the average total mass in each point remains constant for all $t$. This seems surprising since a symmetric random walk is recurrent in dimensions 1 and 2, but it follows by a rescaling argument and the fact that a Brownian motion is point recurrent only in dimension 1. In dimension 1 we give a representation of the asymptotic mass that depends on 
\[a:=\kappa/\rho\]
 but not on the strength of the potential $\gamma$. Let  \[m_{x}(t):=\left\langle u(t,x)\right\rangle,\qquad (t,x)\in  [0,\infty) \times\mathbb{Z}^{d},\]
denote the expected mass at time $t$ in the lattice point $x$. The main results of this section are for $d=1$,
\begin{theorem}\label{homd1}
For all $x\in\mathbb{Z}$,
\[\lim_{t\rightarrow\infty}m_{x}\left(t\right)=1-\frac{1}{\pi}\int\limits^{1}_{0}\text{d}s\,\frac{\sqrt{(1+a)(1-s)s+\frac{a s^2}{1+a}}}{a s^2\left(1+\frac{1}{(1+a)^2}\right)+s},
\]
\end{theorem}
and for higher dimensions
\begin{theorem}\label{homd3}
For $d\geq2$ and all $x\in\mathbb{Z}^{d}$,
	\[\lim_{t\rightarrow\infty}m_{x}(t)=1.
\]
\end{theorem}
\begin{remark}
 Even though the formula in Theorem~\ref{homd1} looks quite clumsy we find that $\lim_{t\to\infty}m_x(t)$ is decreasing in $a$. It tends to $1/2$ as $a$ tends to zero and it tends to zero as $a$ tends to infinity.
\end{remark}

\noindent
The third section is dedicated to analysing the leading order asymptotics of moments of the PAM solution from the perspective of the catalyst, i.e., we consider $\gamma>0$ and the expression $\rsol{t}{x}:=\solu{t}{Y_t+x}$. For $p\in\mathbb{N}$ and $x=(x_1,\ldots,x_p)\in\Zettp$ we denote by
\[
\rmom{t}{x}:=\bigg\langle\prod_{i=1}^p\rsol{t}{x_i}\bigg\rangle
\]
the $p$-th mixed moment at $x$. Moreover, introduce the $p$-th Hamilton operator on $l^\infty\klamm{\Zettp}$ by
\[
\mathcal{H}^p:=\mathcal{A}^p+\gamma V^p
\]
where the potential $V^p$ is defined as $(V^pf)(x)=\sum_{i=1}^p\delta_0(x_i)f(x)$, and $\mathcal{A}^p$ acts on $l^\infty\klamm{\Zettp}$ as
\[
\mathcal{A}^pf(x)=\kappa\sum_{\stackrel{e\in\Zettp}{\betrag{e}=1}}\klamm{f(x+e)-f(x)}+\rho\sum_{\stackrel{e\in\Zett}{\betrag{e}=1}}\klamm{f(x_1+e,\ldots,x_p+e)-f(x)}.
\]
Here, the first term represents the random movement of a collection of $p$ independent random walks accounting for particle diffusion, and the second term arises from the shift by the position of the catalyst. By application of the well-established Feynman-Kac formula and calculating the generator of the resulting semigroup, we obtain the operator representation
\begin{equation}\label{SemigroupRep}
\rmom{t}{x}=\klamm{\expo^{t\mathcal{H}^p}\mathds{1}}(x),\qquad x\in\Zettp.
\end{equation}
This gives the connection between large time moment asymptotics and spectral analysis of the above Hamiltonian. Let us denote by $\lambda_p$ the supremum of the $l^2$-spectrum of $\mathcal{H}^p$. G\"artner and Heydenreich \cite{GH06} have shown that, for all $p\in\mathbb{N}$ and independently of $x\in\Zett$,
\[
\lim_{t\to\infty}(1/t)\log\spitz{u(t,x)^p}=\lambda_p.
\]
This limit is called $p$-th \emph{Lyapunov exponent}. It can be shown by similar methods that just as well
\[
\lim_{t\to\infty}(1/t)\log\rmom{t}{x}=\lambda_p,\qquad x\in\Zettp.
\]
However, this does not enable us to derive large time asymptotics up to equivalence. Assuming the existence of an eigenfunction $(v_p)$ corresponding to $\lambda_p$ with certain properties, we could on a heuristic level decompose the right hand side of equation (\ref{SemigroupRep}) as
\[
\rmom{t}{x}=\mathrm{e}^{t\lambda_p}(\mathds{1},v_p)_{l^2}v_p(x)+o(\mathrm{e}^{t\lambda_p}),\qquad x\in\Zettp.
\]
Our next main result contains criteria under which this is indeed possible.
\begin{theorem}\label{Asymptotics}
Fix $\kappa>0$, $\rho>0$ and let one of the following conditions be satisfied:
\begin{itemize}
 \item[(i)]$p=1$ or $p=2$, $\gamma$ large enough to ensure $\lambda_p>0$,
 \item[(ii)]$p\in\mathbb{N}$, $\gamma>4d\klamm{\kappa p+\rho}$.
\end{itemize}
Then, there exists a strictly positive and summable $l^2$-eigenfunction $v_p$ of $\mathcal{H}^p$ corresponding to $\lambda_p>0$. Assuming $v_p$ to be normed in $l^2\klamm{\Zettp}$, the large time asymptotics of the $p$-th moment are given by

\begin{equation}
\rmom{t}{x}\sim \expo^{\lambda_pt}\,v_p(x)\,\einsnorm{v_p},\qquad t\to\infty,
\end{equation}
where $\einsnorm{\cdot}$ denotes the norm in $l^1\klamm{\Zettp}$.
\end{theorem}

\begin{remark}
In the case $p=1$, $\lambda_p$ is strictly positive if and only if $1/\gamma<G_{\kappa+\rho}(0)$. In this case, the existence of a suitable eigenfunction has been known for quite a while, see e.g. \cite{CM94} or \cite{GdH06}.
\end{remark}
\begin{remark}
For the cases $p\geq2$, the condition $1/\gamma<pG_{\kappa+\rho}(0)$ is sufficient to have positive exponential growth (i.e., $\lambda_p>0$). The condition $\gamma>4d\klamm{\kappa p+\rho}$ also implies exponential growth of the $p$-th moment.
\end{remark}

\section{Moving trap}

This section is devoted to the case $\gamma<0$. Our main proof tool is the Feynman-Kac representation of the solution $u$ given by
$$u(t,x)=\mathbb{E}^X_x\exp\Bigg\{\gamma\int\limits^{t}_{0}\delta_{Y_{t-s}}\left(X_{s}\right)\,\text{d}s\Bigg\}u_0(X_t),\qquad (t,x)\in [0,\infty)\times\mathbb{Z}^{d}.$$
$\E^X_z$, $\E^Y_z$ and $\E^Z_z$ denote the expectation of a random walk with generator $\kappa\Delta$, $\rho\Delta$ and $(\kappa+\rho)\Delta$, respectively. The subscript $z$ indicates the starting point and the corresponding probability measures will be denoted by $\P^\cdot_z$. By 
$$p_t(z)=\P^X_0\left(X_{t/\kappa}=z\right)=\P^X_z\left(X_{t/\kappa}=0\right)$$ we denote the transition probability of a random walk with generator $\Delta$.

\subsection{Localized initial condition}\label{seclic}
In this section we prove Theorems~\ref{locd1} and \ref{locd3}. 
With the help of the Feynman-Kac representation and a time reversal we find that, for all $t\geq0$ and $z\in\Z^d$,
$$M_{z}(t)=\mathbb{E}^{X}_z\E^Y_0 \exp\Bigg\{\gamma\int\limits^{t}_{0}\delta_{0}\left(X_{s}-Y_{s}\right)\,\text{d}s\Bigg\}=\mathbb{E}^{Z}_{z}\exp\Bigg\{\gamma\int\limits^{t}_{0}\delta_{0}\left(Z_{s}\right)\,\text{d}s\Bigg\}.$$

\subsubsection{Dimensions 1 and 2}\label{seclic1}
We start with the dimensions where the random walk is recurrent. 
\begin{proof}[Theorem~\ref{locd1}]
Using the semi-group representation of the resolvent\\ $(\lambda-(\kappa+\rho)\Delta)^{-1}$ we find that
\begin{eqnarray*}
 r_{\lambda}^{\kappa+\rho}(z)&:=&\int\limits^{\infty}_{0}\text{d}t\,{\rm e}^{-\lambda t}\mathbb{E}^{Z}_{z}\exp\Bigg\{\gamma\int\limits^{t}_{0}\delta_{0}\left(Z_{s}\right)\,\text{d}s\Bigg\}\\
&=&\frac{1}{\lambda}+\gamma\Bigg(\int\limits^{\infty}_{0}\text{d}t\,{\rm e}^{-\lambda t}p_{(\kappa+\rho)t}(z)\Bigg)r_{\lambda}^{\kappa+\rho}(0).
\end{eqnarray*}
This implies, for all $\lambda>0$,
$$\int\limits^{\infty}_{0}\text{d}t\,{\rm e}^{-\lambda t}M_{0}(t) = \Bigg(\lambda\Big(1-\gamma\int\limits^{\infty}_{0}\text{d}t\,{\rm e}^{-\lambda t}p_{(\kappa+\rho)t}(0)\Big)\Bigg)^{-1}. $$
Now the claim for $z=0$ follows by a standard Tauberian theorem. The case $z\neq0$ follows due to the recurrence of $Z$. \qed
\end{proof}

\subsubsection{Dimensions 3 and higher}\label{seclic3}
A Tauberian theorem is not applicable in transient dimensions because here the expected number of particles does not converge to zero.
\begin{proof}[Theorem~\ref{locd3}]
 Let
$$v(z):=\lim_{t\rightarrow\infty}M_{z}(t)= \mathbb{E}^{Z}_{z}\exp\Bigg\{\gamma\int\limits^{\infty}_{0}\delta_{0}\left(Z_{s}\right)\,\text{d}s\Bigg\},\qquad z\in\mathbb{Z}^{d}.$$
Notice that the Green's function $G_{\kappa+\rho}$ is finite in transient dimensions and admits the following probabilistic representation.
$$G_{\kappa+\rho}(z)=\mathbb{E}^{Z}_{z}\int\limits^{\infty}_{0}\delta_{0}\left(Z_{s}\right)\,\text{d}s,\qquad z\in\mathbb{Z}^{d}.$$
That implies $v(z)\in(0,1)$ for all $z\in\Z^d$. Furthermore,
we find that $v$ is the unique solution to following boundary problem
\begin{eqnarray*}
\begin{cases}
(\kappa+\rho) \Delta v(z) + \gamma\delta_{0}(z)v(z)=0,
 \qquad &z\in\mathbb{Z}^{d},\\ 
 \lim_{|z|\rightarrow\infty} v(z)=1.
\end{cases}
\end{eqnarray*}
Hence, for all $z\in\Z^d$,
$$v(z)=1+\frac{\gamma}{\kappa+\rho-\gamma G_{1}(0)}G_{1}(z).$$
\qed
\end{proof}

\subsection{Homogeneous initial condition}\label{sechom}
In this section we prove Theorems~\ref{homd1} and \ref{homd3}. For a homogeneous initial condition the Feynman-Kac representation yields, for all $t\geq0$ and $x\in\Z^d$,
$$m_{x}(t)=\sum\limits_{y\in\mathbb{Z}^{d}}\mathbb{E}^{X}_x\E^Y_y\exp\Bigg\{\gamma\int\limits^{t}_{0}\delta_{0}\left(X_{s}-Y_{s}\right)\,\text{d}s\Bigg\}\delta_{0}\left(Y_{t}\right).$$
\subsubsection{Dimension 1}
Let $\tau:=\inf\left\{t\geq0\colon X_{t}=Y_{t}\right\}=\inf\left\{t\geq0\colon Z_{t}=0\right\}\index{$\tau$}$ be the first hitting time of $X$ and $Y$. The density of $\tau$ with respect to $\P^Z_z$, $z\neq0$, will be denoted by $f_\tau^z$.
To prove Theorem~\ref{homd1}, we split $m_x(t)$ into two parts
$$\widetilde{m}_{x}(t):=\sum\limits_{y\in\mathbb{Z}^{d}}\mathbb{E}^{X}_x\E^Y_y\mathbbm{1}_{\tau>t}\delta_{0}(Y_{t}),
$$
where $X$ and $Y$ have not met up to time $t$, and
$$
\widehat{m}_{x}(t):=\sum\limits_{y\in\mathbb{Z}^{d}}\mathbb{E}^{X}_x\E^Y_y\mathbbm{1}_{\tau\leq t}\exp\Bigg\{\gamma\int\limits^{t}_{0}\delta_{0}\left(X_{s}-Y_{s}\right)\,\text{d}s\Bigg\}\delta_{0}\left(Y_{t}\right),$$
where they have already met by time $t$. The next proposition shows that $\widehat{m}_{x}$ is asymptotically negligible. Notice that this implies that there is no difference between the hard trap ($\gamma=-\infty$) and the soft trap ($\gamma\in(-\infty,0)$) case because $\widetilde{m}_{x}$ does not depend on $\gamma$.
\begin{proposition}\label{m1}
 For all $x\in\Z$,
$$\lim_{t\rightarrow \infty}\widehat{m}_{x}(t)=0.$$
\end{proposition}
\begin{proof}
We can assume without loss of generality that $x=0$ since $Z$ is recurrent.
Let $\sigma:=\inf\left\{t\geq0\colon Z_{t}\neq Z_{0}\right\}$ be the the first jumping time of $Z$. Furthermore, for $t\geq0$ let
	$$w\left(t\right):=\mathbb{E}^{Z}_{0}\exp\Bigg\{\gamma\int\limits^{t}_{0}\delta_{0}\left(Z_{s}\right)\,\text{d}s\Bigg\}\delta_{0}\left(Z_{t}\right).$$
In a first step we give an upper bound for the rate of decay of $w$. 
Let us abbreviate $\alpha:=2(\kappa+\rho)$. Using the strong Markov property of $Z$ we find
\begin{eqnarray*}
w\left(t\right)&=&\mathbb{E}^{Z}_{0}\mathbbm{1}_{\sigma>t}{\rm e}^{\gamma t}
+\mathbb{E}^{Z}_{0}\mathbbm{1}_{\sigma\leq t}{\rm e}^{\gamma \sigma}\Bigg[\mathbb{E}^{Z}_{Z_{\sigma}}\exp\Bigg\{\gamma\int\limits^{t-s}_{0}\delta_{0}\left(Z_{u}\right)\,\text{d}u\Bigg\}\delta_{0}\left(Z_{t-s}\right)\Bigg]_{s=\sigma}\\
&=&{\rm e}^{(\gamma-\alpha)t}
+\alpha\int\limits_{0}^{t}\text{d}s\,{\rm e}^{(\gamma-\alpha)s}\mathbb{E}^{Z}_{1}\mathbbm{1}_{\tau\leq t}\Bigg[\mathbb{E}^{Z}_{0}\exp\Bigg\{\gamma\int\limits^{t-s-r}_{0}\delta_{0}\left(Z_{u}\right)\,\text{d}u\Bigg\}\delta_{0}\left(Z_{t-s-r}\right)\Bigg]_{r=\tau}\\
&=&1-E\left(t\right)+\left(\frac{\alpha}{\alpha-\gamma}\left(\psi*f^{1}_{\tau}\right)*w\right)\left(t\right).
\end{eqnarray*}
Here $E$ denotes the distribution function of an exponentially distributed random variable with parameter $\alpha-\gamma$ and $\psi$ denotes the corresponding density.
By iteration we find that, for any $k\geq1$,
$$
w=\left(1-E\right)*\sum\limits_{n=0}^{k}\left(\frac{\alpha}{\alpha-\gamma}\right)^{n}\psi^{*n}*f^{1 *n}_{\tau}+
\left(\frac{\alpha}{\alpha-\gamma}\right)^{(k+1)}\psi^{*(k+1)}*f^{1*(k+1)}_{\tau}*w.
$$
Since there exists $C_1>0$ such that $f^{z}_{\tau}(t)\leq C_1\left(1+t\right)^{-3/2}$ for all $z\neq0$ and $t>0$, we see that asymptotically
\begin{eqnarray*}
w\left(t\right)&\sim&\Bigg(\left(1-E\right)*\Big(\sum\limits_{n=0}^{\infty}\left(\frac{\alpha}{\alpha-\gamma}\right)^{n}\psi^{*n}*f_{\tau}^{1 *n}\Big)\Bigg)\left(t\right)\\
&\leq&\left(1-E(t)\right)*\left(C_{1}\left(1+t\right)^{-3/2}\right)
=C_2\left(1+t\right)^{-3/2},\qquad t\to\infty,
\end{eqnarray*}
where $C_2$ is a positive constant. Let $Z^{(1)}:=X-Y$ and $Z^{(2)}:=X+Y$. Then it follows by H\"older's inequality that
\begin{eqnarray*}
 \widehat{m}_{0}(t)&=&\sum\limits_{z,y\in\mathbb{Z}}\mathbb{E}^{Z^{(1)}}_{-y}, \E^{Z^{(2)}}_{y}\mathbbm{1}_{\tau\leq t}\exp\Bigg\{-\gamma\int\limits^{t}_{0}\delta_{0}\left(Z^{(1)}_{s}\right)\,\text{d}s\Bigg\}\delta_{z}\left(Z^{(1)}_{t}\right)\delta_{z}\left(Z^{(2)}_{t}\right)\\
&\leq&\sum\limits_{z,y\in\mathbb{Z}}\Bigg(\mathbb{E}^{Z^{(1)}}_{-y}\mathbbm{1}_{\tau\leq t}\exp\Bigg\{-\frac{3}{2}\gamma\int\limits^{t}_{0}\delta_{0}\left(Z^{(1)}_{s}\right)\,\text{d}s\Bigg\}\delta_{z}\left(Z^{(1)}_{t}\right)\Bigg)^{2/3}\\
&&\qquad\cdot\bigg(\mathbb{E}^{Z^{(2)}}_{y}\delta_{z}\big(Z^{(2)}_{t}\big)\bigg)^{1/3}.
\end{eqnarray*}
For $t\geq0$ let
\[h\left(t\right):=\mathbb{E}^{Z^{(1)}}_{0}\exp\Bigg\{-\frac{3}{2}\gamma\int\limits^{t}_{0}\delta_{0}\left(Z^{(1)}_{s}\right)\,ds\Bigg\}\delta_{0}\left(Z^{(1)}_{t}\right).\]
Obviously $h$ admits the same asymptotic behaviour as $w$.
Fix $K>0$. The strong Markov property and the central limit theorem yield
\begin{eqnarray*}
 \widehat{m}_{0}(t)&\leq&\sum\limits_{z,y\in\mathbb{Z}}\left(\mathbb{E}^{Z^{(1)}}_{-y}\mathbbm{1}_{\tau\leq t}\left[\mathbb{E}^{Z^{(1)}}_{z}\mathbbm{1}_{\widetilde{\tau}\leq t-s}\Big[h\left(t-s-r\right)
\Big]_{r=\widetilde\tau}\right]_{s=\tau}\right)^{2/3}\\
&&\qquad\cdot\left(p_{2(\kappa+\rho)t}\left(y,z\right)\right)^{1/3}\\
&\sim&\sum\limits_{z,y\leq|K|\sqrt{t}}\left(\mathbb{E}^{Z^{(1)}}_{-y}\mathbbm{1}_{\tau\leq t}\left[\mathbb{E}^{Z^{(1)}}_{z}\mathbbm{1}_{\widetilde{\tau}\leq t-s}\Big[h\left(t-s-r\right)
\Big]_{r=\widetilde\tau}\right]_{s=\tau}\right)^{2/3}\\
&&\qquad\cdot\left(p_{2(\kappa+\rho)t}\left(y,z\right)\right)^{1/3}
\leq CK^{2}t^{-1/6}\stackrel{t\rightarrow\infty}{\longrightarrow}0
\end{eqnarray*}
Here $C$ is a positive constant. This proves the claim .\qed
\end{proof}
Now we show what $\widetilde{m}_{x}$ asymptotically looks like. Recall that $a=\kappa/\rho$.
\begin{proposition}\label{m2}
 For all $x\in\mathbb{Z}$,
$$\lim_{t\rightarrow\infty}\widetilde{m}_{x}\left(t\right)=1-\frac{1}{\pi}\int\limits^{1}_{0}\text{d}s\,\frac{\sqrt{(1+a)(1-s)s+\frac{a s^2}{1+a}}}{a s^2\left(1+\frac{1}{(1+a)^2}\right)+s}.
$$
\end{proposition}
\begin{proof}
 Because of the strong Markov property of $X$ and $Y$ we find
\begin{eqnarray*}
\widetilde{m}_{x}(t)
&=&\sum\limits_{y\in\mathbb{Z}}\mathbb{E}^{X}_x\E^Y_y\delta_{0}(Y_{t}) - \sum\limits_{y\in\mathbb{Z}}\mathbb{E}^{X}_x\E^Y_y\mathbbm{1}_{\tau\leq t}\delta_{0}(Y_{t})\\
&=& 1 - \sum\limits_{y\in\mathbb{Z}}\mathbb{E}^{X}_x\E^Y_y\mathbbm{1}_{\tau\leq t}[\mathbb{E}^{X}_{Y_{\tau}} \E^{Y}_{Y_{\tau}}\delta_{0}(Y_{t-s})]_{s=\tau}\\
&=& 1 - \sum\limits_{y\in\mathbb{Z}}\mathbb{E}^{X}_x\E^Y_y\mathbbm{1}_{\tau\leq t}p_{\rho (t-\tau)}(Y_{\tau}).
\end{eqnarray*}
It follows by Donsker's invariance principle that
$$\lim\limits_{t\to\infty}\sum\limits_{y\in\mathbb{Z}}\mathbb{E}^{X}_x\E^Y_{y}\mathbbm{1}_{\tau\leq t}p_{\rho (t-\tau)}(Y_{\tau})=\int\limits^{\infty}_{-\infty}\text{d}y\,\mathbb{E}^{W^{(1)}}_0\E^{W^{(2)}}_{0}\mathbbm{1}_{\tau_{y}^{(W)}\leq 1}\,p^{({\rm G})}_{\rho (1-\tau_{y}^{(W)})}\left(W^{(2)}_{\tau_{y}^{(W)}}\right).$$
Here $W^{(1)}$ and $W^{(2)}$ denote two independent Brownian motions that start in the origin with variance $2\kappa$ and $2\rho$, respectively.
Their expectations are denoted by $\E^{W^{(1)}}_{0}$ and $\E^{W^{(2)}}_{0}$, respectively. Moreover, $\tau_{y}^{(W)}:=\inf\{t>0\colon W^{(1)}_t-W^{(2)}_t=y\}$ and $p^{({\rm G})}_s$ denotes a Gaussian density with variance $2s$.\\
 Indeed, the application of Donsker's invariance principle is not trivial because we have to sum over all $x\in\Z$, where it cannot be applied uniformly.\\
Let 
$W^{(-)}:=W^{(1)}-W^{(2)}$, $W^{(+)}:=W^{(1)}+\frac{\kappa}{\rho}W^{(2)}$ and $\tau^{(-)}_{y}:=\inf\{t\geq0\colon W^{(-)}_{t}=y\}$.
Notice that $W^{(-)}$ and $W^{(+)}$ are independent. It follows that
\begin{eqnarray*}
&&\int\limits^{\infty}_{-\infty}\text{d}y\,\mathbb{E}^{W^{(1)}}_0\E^{W^{(2)}}_{0}\mathbbm{1}_{\tau_{y}^{(W)}\leq 1}\,p^{({\rm G})}_{\rho (1-\tau_{y}^{(W)})}\left(W^{(2)}_{\tau_{y}^{(W)}}\right)\\
&=&	\int\limits^{\infty}_{-\infty}\text{d}y\,\mathbb{E}^{W^{(-)}}_0\E^{W^{(+)}}_{y}\mathbbm{1}_{\tau_{y}^{(-)}\leq 1}\,p^{({\rm G})}_{\rho (1-\tau_{y}^{(-)})}\left(-\frac{\rho}{\kappa+\rho}W^{(+)}_{\tau_{y}^{(-)}}\right)\\
&=&\int\limits^{\infty}_{-\infty}\text{d}y\,\mathbb{E}^{W^{(-)}}_{0}\mathbbm{1}_{\tau_{y}^{(-)}\leq 1}\,p^{({\rm G})}_{\rho (1-\tau_{y}^{(-)})+\frac{\kappa\rho^{2}\tau_{y}^{(-)}}{(\kappa+\rho)^{2}}}\left(y\right)\\
&=&	\int\limits^{\infty}_{-\infty}\text{d}y\,\int\limits^{1}_{0}\text{d}s\,\frac{|y|\exp\left\{-\frac{y^{2}}{2(\kappa+\rho)s}\right\}}{s\sqrt{2\pi (\kappa+\rho)s} }\frac{\exp\bigg\{-\frac{y^{2}}{2\rho(1-s)+\frac{2\kappa\rho^{2}s}{(\kappa + \rho)^{2}}}\bigg\}}{\sqrt{2\pi \left[2\rho(1-s)+\frac{2\kappa\rho^{2}s}{(\kappa + \rho)^{2}}\right]}}\\
&=&\frac{1}{\pi}\int\limits^{1}_{0}\text{d}s\,\frac{\sqrt{(\kappa+\rho)\rho(1-s)s+\frac{\kappa\rho^{2}s^2}{(\kappa + \rho)}}}{\kappa^2 s^2+\rho s+\frac{\kappa\rho^2 s^2}{(\kappa+\rho)^2}}.
\end{eqnarray*}
Now the claim follows by substituting $a$.\qed
\end{proof}
Theorem~\ref{homd1} follows immediately from Propositions~\ref{m1} and \ref{m2}.

\subsubsection{Dimensions 2 and higher}
In dimensions 2 and higher, we find that asymptotically the expected mass remains constant because a Brownian motion is point recurrent only in dimension 1.
\begin{proof}[Theorem~\ref{homd3}]
Let $\tau_{\varepsilon}^{(Z)}:=\inf\left\{t\geq0\colon Z_{t}\in B_{\varepsilon}(0)\right\}$ be the first time that the process $Z$ hits the centered ball $B_{\varepsilon}(0)$ with radius $\varepsilon>0$, and let 
$$\overline{m}_{x}(t):=\sum\limits_{y\in\mathbb{Z}^{d}}\mathbb{E}^{X}_x \E^{Y}_{y}\mathbbm{1}_{\tau_{\varepsilon\sqrt{t}}^{(Z)}>t}\delta_{0}(Y_{t}).$$
 Similarly as in the case $d=1$ we find with the help of Donsker's invariance principle that
  $$\lim\limits_{t\to\infty}1-\overline{m}_{0}(t)=\lim_{\varepsilon\rightarrow0}\int\limits_{\mathbb{R}^{d}}\text{d}x\,\mathbb{P}^{W}_{x}\left(\tau^{(W)}_{\varepsilon}\leq\frac{1}{2}\right).$$
However, for $d\geq2$ and $x\neq0$,
$$\lim_{\varepsilon\rightarrow0}\mathbb{P}^{W}_{x}\left(\tau^{(W)}_{\varepsilon}\leq\frac{1}{2}\right)=\mathbb{P}^{W}_{x}\left(\bigcap\limits_{\varepsilon>0}\left\{\tau^{(W)}_{\varepsilon}\leq\frac{1}{2}\right\}\right)=\mathbb{P}^{W}_{x}\left(\tau^{(W)}_{0+}\leq\frac{1}{2}\right)=0.$$
Hence, it follows by monotone convergence that $\lim_{t\rightarrow\infty}\overline{m}_{0}(t)=1$ which implies that  $\lim_{t\rightarrow\infty}m_{x}(t)=1$ for all $x\in\Z^d$.\qed
\end{proof}

\section{Moving catalyst}

In this section we stick to the homogeneous initial condition $u_0\equiv1$ and examine the case of a randomly moving catalyst, i.e., we consider $\gamma>0$.

\subsection{Spectral properties of higher-order Anderson Hamiltonians}

Throughout this section, we write $\lambda_p:=\sup\Spectrum{\mathcal{H}^p}$ for all $p\in\mathbb{N}$. Considering the first Hamilton operator $\mathcal{H}^1$ given by

\begin{equation*}
\mathcal{H}^1:=\klamm{\kappa+\rho}\Delta +\gamma\delta_0,
\end{equation*}
the existence of an eigenfunction $v_1\in l^2\klamm{\Zett}$ corresponding to its largest spectral value, provided that this value is greater than zero, has been widely known for some time. The following theorem extends this to the case $p=2$ and constitutes the main statement of this section:
\begin{theorem}\label{SpektrumH2}
Assume $\lambda_2>0$. Then, $\lambda_2$ is isolated in the point spectrum of $\mathcal{H}^2$ with one-dimensional eigenspace. The corresponding eigenfunction may be chosen strictly positive.
\end{theorem}
\noindent
For a start, we restrict the operator to the subspace of component-wise symmetric functions
\begin{equation*}
\Symm{2}:=\KLAMM{f\in l^2\klamm{\Zetttwo}\lvert f(x,y)=f(y,x)\quad\forall x,y\in\Zetttwo},
\end{equation*}
\index{$\Symm{p}$} which is obviously closed in $l^2\klamm{\Zetttwo}$. Recall the definition of the operators $\mathcal{A}^2$ and $V^2$ from Section \ref{Introduction} and define $\Asym{2}$ and $\Vsym{2}$ as their restrictions on the set $\Symm{2}$ above. In the same manner, we denote by $\Hsym{2}$ the restricted second-order Hamilton operator. The reader may easily retrace these operators are endomorphisms on $\Symm{2}$. In particular, $\Hsym{2}$ is a self-adjoint operator on the Hilbert space $\Symm{2}$, and it is essential that the supremum of its spectrum coincides with $\lambda_2$, which can be shown elementarily. Each eigenfunction of $\Hsym{2}$ corresponding to $\lambda_2$ is an eigenfunction of $\mathcal{H}^2$ as well. Moreover, we expect that an eigenfunction of $\mathcal{H}^2$ is, or at least could be chosen as, an element of $\Symm{2}$. In view of that, passing over to $\Symm{2}$ is just a natural approach. In the next step, we write $\Asym{2}+\gamma\Vsym{2}$ rather than $\Hsym{2}$ in order to emphasize the dependence on the potential parameter $\gamma$, and we establish a further translation of the main task:

\begin{lemma}\label{SpectrumDuality} 
Suppose $\lambda>0$. Then, the resolvent operator $\Rsym{\lambda}:=(\lambda-\Asym{2})^{-1}$\index{$\Rsym{\lambda}$} exists on $\Symm{2}$, and for all $\gamma>0$, we have
\begin{itemize}
 \item [(i)] \begin{displaymath}
		\lambda\in\sigma\klamm{\Asym{2}+\gamma\Vsym{2}}\qquad\Longleftrightarrow\qquad\gamma^{-1}\in\sigma\klamm{\Rsym{\lambda}\Vsym{2}},
		\end{displaymath}
 \item [(ii)]\begin{displaymath}
		\lambda=\sup\sigma\klamm{\Asym{2}+\gamma \Vsym{2}}\qquad\Longrightarrow\qquad\gamma^{-1}=\sup\sigma\klamm{\Rsym{\lambda} \Vsym{2}}.
		\end{displaymath}
\end{itemize}
Moreover, for $v\in \Symm{2}$ and $\gamma>0$,
\begin{itemize} 
 \item [(iii)] \begin{equation*}
		\klamm{\Asym{2}+\gamma \Vsym{2}}v=\lambda v\qquad\Longleftrightarrow\qquad\klamm{\Rsym{\lambda} \Vsym{2}}v=\frac{1}{\gamma}v.
		\end{equation*}
\end{itemize}
\end{lemma}
\begin{proof}
A Fourier transform reveals that the spectrum of $\mathcal{A}^2$ is concentrated on the negative half-axis, thus $(\lambda-\mathcal{A}^2)^{-1}$ exists on $l^2\klamm{\Zetttwo}$ for all $\lambda>0$. In particular, it exists on $\Symm{2}$, and then it coincides with $(\lambda-\Asym{2})^{-1}$ as $\Asym{2}$ is an endomorphism on $\Symm{2}$. Assertions (i) and (iii) follow by rearranging the equations considered and applying the resolvent operator. The second relation is shown using the Rayleigh-Ritz formula.\qed
\end{proof}
\noindent
As a next step, we introduce an operator $\TS{\lambda}$ on $l^2\klamm{\Zett}$ having the same spectrum and the same point spectrum as $\Rsym{\lambda}\Vsym{2}$ and that admits the decomposition $\TS{\lambda}=\TSi{1}+\TSi{2}$. Here, $\TSi{1}$ is compact and the supremum of $\sigma(\TSi{2})$ is strictly smaller than the supremum of $\sigma(\TS{\lambda})$. Then, we use Weyl's theorem to obtain that the largest value in $\sigma(\TS{\lambda})$ belongs to the point spectrum $\sigma_{\rm p}(\TS{\lambda})$. The resolvent $R_\lambda:=\klamm{\lambda-\mathcal{A}^2}^{-1}$ admits the representation

\begin{equation*}
\klamm{R_\lambda f}(x_1,x_2)=\sum_{y_1,y_2\in\Zett}\rk{2}{\lambda}(y_1-x_1,y_2-x_2)f(y_1,y_2),\quad x_1,x_2\in\Zett,
\end{equation*}
where the resolvent kernel $\rk{2}{\lambda}:\Zetttwo\to(0,\infty)$ is defined as

\begin{equation*}
\rk{2}{\lambda}(x_1,x_2):=\int_0^\infty \text{d}t\,\mathrm{e}^{-\lambda t}\mathbb{P}_0\klamm{Z_t=(x_1,x_2)},\quad x_1,x_2\in\Zett.
\end{equation*}
Here, $Z$ is a random walk on $\Zetttwo$ with generator $\mathcal{A}^2$. Then we obtain

\begin{equation}
\klamm{R_\lambda V^2 f}(x_1,x_2)=\sum_{y_1,y_2\in\Zett}\rk{2}{\lambda}(y_1-x_1,y_2-x_2)\Klamm{\delta_0(y_1)+\delta_0(y_2)}f(y_1,y_2).
\end{equation}
If we assume $f\in\Symm{2}$, we get

\begin{equation*}
\klamm{\Rsym{\lambda} \Vsym{2} f}(x_1,x_2)=\sum_{y\in\Zett}\Klamm{\rk{2}{\lambda}(y-x_1,-x_2)+\rk{2}{\lambda}(-x_1,y-x_2)}f(y,0),
\end{equation*}
for $x_1,x_2\in\Zett$, and in particular

\begin{equation*}
\klamm{\Rsym{\lambda} \Vsym{2} f}(x,0)=\sum_{y\in\Zett}\Klamm{\rk{2}{\lambda}(y-x,0)+\rk{2}{\lambda}(-x,y)}f(y,0),
\end{equation*}
for $x\in\Zett$. Let us therefore introduce the operator $\TS{\lambda}=\TSi{1}+\TSi{2}$\index{$\TS{\lambda}$} on $l^2\klamm{\Zett}$ with\index{$\TSi{1},\TSi{2}$}

\begin{equation*}
\TSi{1}\tilde{f}(x):=\sum_{y\in\Zett}\rk{2}{\lambda}(-x,y)\tilde{f}(y),\quad\TSi{2}\tilde{f}(x):=\sum_{y\in\Zett}\rk{2}{\lambda}(y-x,0)\tilde{f}(y),\quad x\in\Zett.
\end{equation*}
Both operators are apparently self-adjoint. The lemma below identifies the spectra and point spectra of $\TS{\lambda}$ and $\Rsym{\lambda}\Vsym{2}$:

\begin{lemma}\label{SpektrumTRV}
For all $\lambda>0$,
\begin{equation*}
\Spectrum{\TS{\lambda}}=\Spectrum{\Rsym{\lambda} \Vsym{2}},\quad\PSpectrum{\TS{\lambda}}=\PSpectrum{\Rsym{\lambda} \Vsym{2}}.
\end{equation*}
\end{lemma}
\begin{proof}
The crucial and least intuitive part is to show that 
\begin{equation}\label{Surjective}
\mu-\TS{\lambda}\text{ surjective}\Rightarrow \mu-\Rsym{\lambda}\Vsym{2}\text{ surjective.}
\end{equation}
All other implications are rather straightforward and we omit them for the sake of conciseness. Assume $\mu-\TS{\lambda}$ is surjective and choose $g\in \Symm{2}$. Define $\tilde{g}(x):=g(x,0)$ for $x\in\Zett$. There exists $\tilde{f}\in l^2\klamm{\Zett}$ with $\klamm{\mu-\TS{\lambda}}\tilde{f}=\tilde{g}$ by assumption. We define

\begin{equation*} 
f(x_1,x_2):=f(x_1)\delta_0(x_2)+f(x_2)\delta_0(x_1)-\delta_0(x_1)\delta_0(x_2)f(0),\qquad x_1,x_2\in\Zett
\end{equation*}
and then, for $x_1,x_2\in\Zett$,

\begin{equation*} 
F(x_1,x_2):=
\begin{cases}
f(x_1,x_2),\quad &x_1=0 \text{ or } x_2=0;\\
\mu^{-1}\klamm{\Rsym{\lambda}\Vsym{2}f(x_1,x_2)+g(x_1,x_2)},\quad &\text{else.}
\end{cases}
\end{equation*}
We realize that $F\in\Symm{2}$ and proceed by showing that $F$ is the desired function satisfying $\klamm{\mu-\Rsym{\lambda}\Vsym{2}}F=g$. Note that $\Rsym{\lambda}\Vsym{2}F(x_1,x_2)=\Rsym{\lambda}\Vsym{2}f(x_1,x_2)$ for all $x_1,x_2\in\Zett$. In the first place, we have

\begin{align}\label{Surjective1}
\Rsym{\lambda}\Vsym{2}F(x_1,0)&=\sum_{y\in\Zett}\Klamm{\rk{2}{\lambda}(y-x_1,0)+\rk{2}{\lambda}(-x_1,y)}f(y,0)\notag\\
&=\TS{\lambda}\tilde{f}(x_1)=\mu\tilde{f}(x_1)-\tilde{g}(x_1)\notag\\
&=\mu F(x_1,0)-g(x_1,0),\qquad x_1\in\Zett,
\end{align}
and by symmetry $\Rsym{\lambda}\Vsym{2}F(0,x_2)=\mu F(0,x_2)-g(0,x_2)$ for $x_2\in\Zett$. Moreover,
\begin{align}\label{Surjective2}
&\quad\mu F(x_1,x_2)-\Rsym{\lambda}\Vsym{2}F(x_1,x_2)\notag\\
=&\quad\Rsym{\lambda}\Vsym{2}f(x_1,x_2)+g(x_1,x_2)-\Rsym{\lambda}\Vsym{2}f(x_1,x_2)\notag\\
=&\quad g(x_1,x_2),\qquad x_1,x_2\in\Zett,x_1,x_2\neq 0.
\end{align}
Equations (\ref{Surjective1}) and (\ref{Surjective2}) yield the desired result $(\mu-\Rsym{\lambda}\Vsym{2})F=g$. Thus, we have shown (\ref{Surjective}).\qed
\end{proof}

\noindent
In the next step, we are able to calculate the supremum of the spectrum of $\TSi{2}$. Its value is given in terms of the Laplace resolvent kernel $r^\kappa_\lambda$\index{$r^{\kappa}_\lambda$} defined by

\begin{equation}\label{ResolventKernel1}
r^{\kappa}_\lambda(x):=\int_0^\infty \text{d}t\,\mathrm{e}^{-\lambda t}\mathbb{P}_0\klamm{X_t=x},\quad x\in\Zett,
\end{equation}
with $X$ a random walk on $\Zett$ with generator $\kappa\Delta$.

\begin{lemma}\label{SpektrumT2}
We have $\sup\sigma(\TSi{2})=\Vert\TSi{2}\Vert_2=r^{\kappa+\rho}_\lambda(0)$.
\end{lemma}

\begin{proof}
It will be sufficient to show that 
\begin{equation}\label{SpectralRadius}
\sup\KLAMM{\betrag{\mu}\colon\mu\in\sigma(\TSi{2})}=r^{\kappa+\rho}_\lambda(0).
\end{equation}
The proof involves a Fourier transform (which we denote by $\mathcal{F}$) of the operator $\TSi{2}$. For $\hat{f}\in L^2\klamm{[-\pi,\pi)^d}$, the transformed operator reads

\begin{align*}
\hat{T}^{(2)}\hat{f}(l)&=(2\pi)^{-d}\sum_{x\in\Zett}\mathrm{e}^{{\rm i}(l,x)}\sum_{y\in\Zett}\rk{2}{\lambda}(y-x,0)\int_{\klamm{[-\pi,\pi)^d}} \text{d}k\,\mathrm{e}^{-{\rm i}(k,y)}\hat{f}(k)\notag\\
&=(2\pi)^{-d}\sum_{y\in\Zett}\mathrm{e}^{{\rm i}(l,y)}\int_{\klamm{[-\pi,\pi)^d}} \text{d}k\,\mathrm{e}^{-{\rm i}(k,y)}\hat{f}(k)\sum_{y\in\Zett}\mathrm{e}^{{\rm i}(l,x-y)}\rk{2}{\lambda}(x-y,0)\notag\\
&=\klamm{\mathcal{F}\mathcal{F}^{-1}\hat{f}}(l)\klamm{\sum_{z\in\Zett}\mathrm{e}^{{\rm i}(l,z)}\rk{2}{\lambda}(z,0)},\qquad l\in[-\pi,\pi)^d.
\end{align*}
Thus, $\hat{T}^{(2)}$ is a multiplication operator and the multiplier

\begin{equation*}
\hat{r}(l):=\sum_{z\in\Zett}\mathrm{e}^{{\rm i}(l,z)}\rk{2}{\lambda}(z,0),\qquad l\in[-\pi,\pi)^d
\end{equation*}
is obviously continuous. Hence, its spectrum is just the closure of the range of that multiplier. As each of the two components of a random walk on $\Zetttwo$ with generator $\mathcal{A}^2$ is just a random walk on $\Zett$ with generator $(\kappa+\rho)\Delta$, we have

\begin{equation*}
\sup_{l\in[-\pi,\pi)^d}\betrag{\hat{r}(l)}=\hat{r}(0)=\sum_{z\in\Zett}\rk{2}{\lambda}(z,0)=r^{\kappa+\rho}_\lambda(0),
\end{equation*} 
and equation (\ref{SpectralRadius}) follows by taking into account that the Fourier transform is an isometry.\qed
\end{proof}

\begin{lemma}\label{SpectrumTlambda}
Suppose $\lambda_2=\sup\sigma\klamm{\mathcal{H}^2}$. Then, the operator $\TS{\lambda_2}=\TSi{1}+\TSi{2}$ has a strictly positive eigenfunction $\tilde{v}$ corresponding to its largest spectral value $1/\gamma$. This value is isolated in the spectrum.
\end{lemma}
\begin{proof}
At first, we realize that $\TSi{1}$ belongs to the trace class as
\begin{equation*}
\sum_{x\in\Zett} \klamm{\TSi{1}\delta_x,\delta_x}<\int_0^\infty \text{d}t\,\mathrm{e}^{-\lambda_2 t}<\infty,
\end{equation*}
and therefore $\TSi{1}$ is compact. Then, we explain why $\sup\sigma(\TSi{1}+\TSi{2})>\sup\sigma(\TSi{2})$, which together with Weyl's theorem (see e.g. \cite{RS72}) yields the existence of an eigenfunction. In the end, it remains to show that we may choose this eigenfunction strictly positive.\vspace{5pt}\\
\noindent
In order to show that $\sup\sigma(\TSi{1}+\TSi{2})>\sup\sigma(\TSi{2})$, we recall that

\begin{equation*}
\sup\sigma\klamm{\TS{\lambda_2}}=\sup\sigma\klamm{\Rsym{\lambda_2}\Vsym{2}}=\frac{1}{\gamma}
\end{equation*}
by Lemmas \ref{SpectrumDuality} and \ref{SpektrumTRV}, and the supremum of $\sigma(\TSi{2})$ is equal to $r^{\kappa+\rho}_{\lambda_2}(0)$ by Lemma \ref{SpektrumT2}. Therefore, it suffices to show that

\begin{equation}\label{contraSpectrum}
\frac{1}{\gamma}>r^{\kappa+\rho}_{\lambda_2}(0).
\end{equation}
Let $\lambda_1:=\sup\sigma\klamm{\mathcal{H}^1}$. In case $\lambda_1>0$, it is well-known that

\begin{equation*}
\frac{1}{\gamma}=r^{\kappa+\rho}_{\lambda_1}(0),
\end{equation*}
compare e.g. Carmona and Molchanov \cite{CM94}. Moreover, as $\lambda_1$ and $\lambda_2$ are the exponential growth rates of the first and second moment of $\rsol{t}{x}$, H\"older's inequality yields

\begin{equation*}
\lambda_1\leq\frac{1}{2}\lambda_2,
\end{equation*}
thus a fortiori $\lambda_1<\lambda_2$. As $\lambda\mapsto r^{\kappa+\rho}_{\lambda}(0)$ is strictly decreasing (see e.g. equation (\ref{ResolventKernel1})),

\begin{equation*}
\frac{1}{\gamma}=r^{\kappa+\rho}_{\lambda_1}(0)>r^{\kappa+\rho}_{\lambda_2}(0)
\end{equation*}
and we have shown (\ref{contraSpectrum}) for the case $\lambda_1>0$. In case $\lambda_1=0$, we have

\begin{equation*}
\frac{1}{\gamma}\geq G_{\kappa+\rho}(0),
\end{equation*}
and we arrive at (\ref{contraSpectrum}) as  $G_{\kappa+\rho}(0)>r^{\kappa+\rho}_{\lambda}(0)$ for all $\lambda>0$. Weyl's theorem now states that $1/\gamma$ belongs to the discrete spectrum of $\TS{\lambda_2}$ since $\TSi{1}$ is compact. Consequently, the value $1/\gamma$ is isolated in the point spectrum. Finally, we show that a corresponding eigenfunction $\tilde{v}$ may be chosen strictly positive. It suffices to show that $\TS{\lambda_2}$ is positive in the sense that it maps nonnegative, non-zero functions to positive functions. Choose a nonnegative function $f$ arbitrarily and assume $f(y_1)>0$ for some $y_1\in\Zett$. Then, for all $x\in\Zett$,

\begin{equation*}
\TS{\lambda_2}f(x)\geq \Klamm{\rk{2}{\lambda_2}(y_1-x,0)+\rk{2}{\lambda_2}(-x,y_1)}f(y_1)>0.
\end{equation*}
Consequently, we may choose $\tilde{v}$ strictly positive, and the proof is complete.\qed
\end{proof}
\noindent
Let us now prove the main result of this section:

\begin{proof}[Theorem \ref{SpektrumH2}]
Let $\lambda_2=\sup\sigma\klamm{\mathcal{H}^2}$. The preceding lemma states that there exists a strictly positive function $\tilde{v}\in l^2\klamm{\Zett}$ with $T_{\lambda_2}\tilde{v}=(1/\gamma)\tilde{v}$. By Lemma \ref{SpektrumTRV}, there exists $v_2\in \Symm{2}$ with $\Rsym{\lambda_2}\Vsym{2}v_2=(1/\gamma) v_2$, as point spectra of both operators coincide. Naturally, $v_2$ is also an eigenfunction of $R_{\lambda_2}V^p$ on $l^2\klamm{\Zetttwo}$. We easily verify that $$R_{\lambda_2}V^pf>0$$ for all nonnegative, non-zero $f\in l^2\klamm{\Zetttwo}$, thus $v_2$ may be chosen strictly positive. Now Lemma \ref{SpectrumDuality} yields that $v_2$ is an eigenfunction of $\Hsym{2}$ and $\mathcal{H}^2$ corresponding to $\lambda_2$.\vspace{5pt}\\
\noindent
In order to show that its corresponding eigenspace is one-dimensional, let $\klamm{w_i}_{i\in I}$ represent an orthonormal basis of this eigenspace. The $w_i$ are principal eigenfunctions of $R_{\lambda_2}V^p$ that maps nonnegative, non-zero functions to positive functions. Hence we may choose all $w_i$ strictly positive. As two strictly positive functions in $l^2\klamm{\Zettp}$ cannot be orthogonal, it follows that $\betrag{I}=1$, i.e., the eigenspace corresponding to $\lambda_2$ is one-dimensional.\qed
\end{proof}
We will additionally need that the largest eigenvalue $\lambda_2$ is isolated in the spectrum of $\Hsym{2}$ in order to describe the asymptotic moment behaviour:

\begin{lemma}\label{IsolatedSymm}
The value $\lambda_2$ is isolated in $\sigma(\Hsym{2})$.
\end{lemma}

\begin{proof}
We know by Lemma \ref{SpectrumTlambda} that $\gamma^{-1}=\sup\sigma(\TS{\lambda_2})$ is an isolated eigenvalue, so there exists $\tilde{\delta}>0$ with 

\begin{equation*}
\Klamm{\gamma^{-1}-\tilde{\delta},\gamma^{-1}+\tilde{\delta}}\cap\sigma\klamm{\TS{\lambda_2}}=\KLAMM{\gamma^{-1}}.
\end{equation*}
Define now $\delta_1$ small enough to ensure

\begin{equation*}\label{Isolated1}
\norm{\TS{\lambda_2-\varepsilon}-\TS{\lambda_2}}<\tilde{\delta}/2 \text{ for all }\varepsilon\text{ with }0<\varepsilon\leq\delta_1.
\end{equation*}
It is quickly verified that this is always possible (e.g. by the mean value theorem). We can show with a similar argument that $\sup\sigma(\TS{\lambda})$ depends continuously on $\lambda$, making it possible to find $\delta_2$ small enough to satisfy

\begin{equation*}
\sup\sigma\klamm{\TS{\lambda_2-\varepsilon}}-\sup\sigma\klamm{\TS{\lambda_2}}\leq\tilde{\delta}/2 \text{ for all }\varepsilon\text{ with }0<\varepsilon\leq\delta_2.
\end{equation*}
If we choose now $\varepsilon<\delta_1\wedge\delta_2$, it follows

\begin{equation*}
\gamma^{-1}\neq\sup\sigma\klamm{\TS{\lambda_2-\varepsilon}}\in\Klamm{\gamma^{-1}-\tilde{\delta}/2,\gamma^{-1}+\tilde{\delta}/2},
\end{equation*}
and by Theorem \ref{perturbation} below, the interval $[\gamma^{-1}-\tilde{\delta}/2,\gamma^{-1}+\tilde{\delta}/2]$ contains exactly one element of the spectrum of $\TS{\lambda_2}+(\TS{\lambda_2-\varepsilon}-\TS{\lambda_2})=\TS{\lambda_2-\varepsilon}$. Therefore, $\gamma^{-1}\in\rho(\TS{\lambda_2-\varepsilon})$ and it follows that $\lambda_2-\varepsilon\in\rho(\Hsym{2})$ in the usual way by Lemmas \ref{SpectrumDuality} and \ref{SpektrumTRV}. Thus, $\lambda_2$ is isolated in $\sigma(\Hsym{2})$.\qed
\end{proof}
\noindent
Let us in the following present a sufficient condition for the existence of an eigenfunction with the desired properties that holds for general $p\in\mathbb{N}$:

\begin{theorem}\label{SpektrumHpLarge}
Suppose $p\in\mathbb{N}$ and $\gamma>4d\klamm{\kappa p+\rho}$. Then, $\lambda_p=\sup\Spectrum{\mathcal{H}^p}$ is positive, isolated in the spectrum and belongs to the point spectrum with one-dimensional eigenspace. The corresponding eigenfunction may be chosen strictly positive.
\end{theorem}
\noindent
The proof relies on the following theorem from perturbation theory of bounded operators. It describes the behaviour of an isolated eigenvalue under a bounded perturbation that is sufficiently small in a certain sense, see Birman and Solomjak \cite{BS80}, Ch. 9.4 for a proof.

\begin{theorem}\label{perturbation}
Let $T,S$ denote two self-adjoint operators on a Hilbert space. Suppose $\mu\in\PSpectrum{T}$ with multiplicity $r$ and
\begin{equation*}
\Klamm{\mu-\varepsilon,\mu+\varepsilon}\cap \sigma\klamm{T}=\KLAMM{\mu}
\end{equation*}
for some $\varepsilon>0$. Moreover, assume $\Spectrum{S}\subset\Klamm{\delta_1,\delta_2}$ for some $\delta_1<\delta_2\in\mathbb{R}$ with $\delta_2-\delta_1<\varepsilon$. Then, the set
\begin{equation*}
\Klamm{\mu+\delta_1,\mu+\delta_2}\cap\Spectrum{T+S}
\end{equation*}
contains only isolated eigenvalues of $T+S$ whose sum of multiplicities equals $r$.
\end{theorem}

\begin{proof}[Theorem \ref{SpektrumHpLarge}]
We have $\mathcal{H}^p=\mathcal{A}^p+\gamma V^p$, and the idea is to understand the generator $\mathcal{A}^p$ as a perturbation of the potential $\gamma V^p$. With increasing $\gamma$, the perturbation $\mathcal{A}^p$ remains relatively small, which allows an application of Lemma \ref{perturbation}. As $\gamma V^p$ is a multiplication, its spectrum coincides with the essential range of the multiplier and we easily verify that $\gamma p$ is the largest eigenvalue of $\gamma V^p$ and has one-dimensional eigenspace. Moreover,

\begin{equation*}
\Spectrum{\gamma V^p}\cap\Klamm{\gamma p-\gamma,\gamma p+\gamma}=\KLAMM{\gamma p},
\end{equation*}
and we may show by a Fourier transform that $$\Spectrum{\mathcal{A}^p}\subset\Klamm{-4d\klamm{\kappa p+\rho},0}.$$ Theorem \ref{perturbation} now yields that the set
\begin{equation*}
\Spectrum{\mathcal{H}^p}\cap\Klamm{\gamma p-\gamma,\gamma p}
\end{equation*}
contains exactly one element, which is an eigenvalue with multiplicity one. This element must be $\lambda_p=\sup\Spectrum{\mathcal{H}^p}$ due to the nonpositive definiteness of $\mathcal{A}^p$. It remains to show that the corresponding eigenfunction may be chosen strictly positive. To that purpose, we consider that $v_p$ is also an eigenfunction of $\mathrm{e}^{\mathcal{H}^p}$ corresponding to its largest eigenvalue $\mathrm{e}^{\lambda_p}$. Employing the Feynman-Kac representation of this operator, we see that it maps nonnegative, non-zero functions to strictly positive functions. This means that all principal eigenfunctions are either strictly positive or strictly negative.\qed
\end{proof}

\subsection{Application to annealed higher moment asymptotics}
A natural approach to more exact asymptotics of mixed moments, and the main idea proving Theorem \ref{Asymptotics}, is to decompose the semigroup representation

\begin{equation*}
\rmom{t}{x}=\klamm{\expo^{t\mathcal{H}^p}\mathds{1}}(x).
\end{equation*}
Certainly we must consider the initial condition $\mathds{1}$ as an appropriate limit of $l^2$-functions when attempting a rigorous proof. With $\KLAMM{E_\mu\vert\mu\in\mathbb{R}}$\index{$E_\mu$} the family of spectral projectors associated with $\mathcal{H}^p$, the spectral theorem for self-adjoint operators yields

\begin{equation}\label{IntrAsymp2}
\rmom{t}{x}=\expo^{\lambda_pt}\klamm{\mathds{1},v_p}v_p(x)+\int_{-\infty}^{\lambda_p-\varepsilon}\expo^{\mu t}\text{d}E_\mu\klamm{\mathds{1}}(x)
\end{equation}
for some $\varepsilon>0$ small enough. Here, $\lambda_p=\sup\Spectrum{\mathcal{H}^p}$ must be a positive eigenvalue with multiplicity one that is isolated in $\Spectrum{\mathcal{H}^p}$, and $v_p$\index{$v_p$} is a strictly positive and $l^2$-normed eigenfunction corresponding to $\lambda_p$. Beyond that, we need that $\klamm{\mathds{1},v_p}<\infty$. If these requirements are met, we may asymptotically neglect the last term in (\ref{IntrAsymp2}). In order to prove Theorem \ref{Asymptotics}, we need the following two auxiliary lemmas that are of pure technical nature and thus given without a proof. The first one enables us to approximate the homogeneous initial condition with $l^2$-functions. Let \mbox{$\mathbf{Q}_{t}^p:=\mathbb{Z}^{pd} \cap \left[-t,t\right]^{pd}$}\index{$\mathbf{Q}_{t}^p$} for $t>0$.

\begin{lemma}\label{QApprox}
For all $x\in\Zettp$,

\begin{equation*}
\klamm{\mathrm{e}^{t\mathcal{H}^p}\mathds{1}_{\mathbf{Q}_{t^2}^p}}(x)\sim\rmom{t}{x},\qquad t\to\infty.
\end{equation*}
\end{lemma}
\noindent
The second auxiliary lemma ensures that the considered principal eigenfunctions are summable:

\begin{lemma}\label{Summable}
Suppose $\lambda_p=\sup\Spectrum{\mathcal{H}^p}>0$ and there exists a corresponding eigenfunction $v_p\in l^2\klamm{\Zettp}$. Then, $v_p\in l^1\klamm{\Zettp}$.
\end{lemma}
\noindent
Let us now give a concise proof of the main statement:

\begin{proof}[Theorem \ref{Asymptotics}]
It suffices to show that, for all $x\in\Zettp$,
\begin{equation}\label{CompLimit}
\mathrm{e}^{-\lambda_p t}\rmom{t}{x}\longrightarrow \einsnorm{v_p}v_p(x)
\end{equation}
as $t$ approaches infinity. The spectral representation of $\mathrm{e}^{t\mathcal{H}^p}$ yields

\begin{align*}
\mathrm{e}^{-t\lambda_p+t\mathcal{H}^p}\mathds{1}_{\mathbf{Q}^p_{t^2}}=\klamm{\mathds{1}_{\mathbf{Q}^p_{t^2}},v_p}v_p+\int_{-\infty}^{\lambda_p-\varepsilon}\mathrm{e}^{t(\mu-\lambda_p)}\,\text{d}E_\mu\klamm{\mathds{1}_{\mathbf{Q}^p_{t^2}}},
\end{align*}
for some $\varepsilon>0$ as $\lambda_p$ is isolated in the spectrum. For $t$ large enough, the $l^2$-norm of the integral is roughly estimated from above by

\begin{equation*}
\norm{\int_{-\infty}^{\lambda_p-\varepsilon}\mathrm{e}^{t\mu-t\lambda_p}\,\text{d}E_\mu\klamm{\mathds{1}_{\mathbf{Q}^p_{t^2}}}}\leq\mathrm{e}^{-t\varepsilon}\norm{\mathds{1}_{\mathbf{Q}^p_{t^2}}}\leq \klamm{2t^2}^{\frac{pd}{2}}\mathrm{e}^{-t\varepsilon},
\end{equation*}
which means we may neglect this term and have

\begin{equation}\label{Convl2}
\mathrm{e}^{-t\lambda_p+t\mathcal{H}^p}\mathds{1}_{\mathbf{Q}^p_{t^2}}\overset{l^2}{\longrightarrow}\einsnorm{v_p}v_p,\qquad t\to\infty.
\end{equation}
The limit (\ref{CompLimit}) now follows by equation (\ref{Convl2}), Lemma \ref{QApprox} and the triangle inequality. This completes the proof.\qed
\end{proof}

\subsection*{Acknowledgement}
The results of this paper have been derived in two theses under the supervision of J\"urgen G\"artner whom we would like to thank for his invaluable support.

\end{document}